\documentclass[11pt]{amsart}
\usepackage[top=1.5in, bottom=1.5in, left=1.4in, right=1.4in]{geometry}                
\usepackage{color}
\usepackage{graphicx}
\usepackage{amssymb}
\usepackage{epstopdf}
\usepackage{amsthm}
\usepackage{hyperref}
\newtheorem{theorem}{Theorem}[section]
\newtheorem{lemma}[theorem]{Lemma}
\newtheorem{definition}[theorem]{Definition}
\newtheorem{corollary}[theorem]{Corollary}

\title{Partial Graph Orientations and the Tutte Polynomial}
\author{Spencer Backman}
\keywords{Partial graph orientation, Tutte polynomial, Potts model, $V$-polynomial cycle-cocycle reversal system, Lawrence ideal, reliability polynomial. }

\def\O{\mathcal{O}}

\begin{document}
\maketitle

\bibliographystyle{plain}

\begin{abstract}

Gessel and Sagan \cite{gessel1996tutte} investigated the Tutte polynomial, $T(x,y)$ using depth first search, and applied their techniques to show that the number of acyclic partial orientations of a graph is $2^gT(3,1/2)$.  We provide a short deletion-contraction proof of this result and demonstrate that dually, the number of strongly connected partial orientations is $2^{n-1}T(1/2,3)$.  We then prove that the number of partial orientations modulo cycle reversals is $2^gT(3,1)$ and the number of partial orientations modulo cut reversals is $2^{n-1}T(1,3)$.  To prove these results, we introduce cut and cycle minimal partial orientations which provide distinguished representatives for partial orientations modulo cut and cycle reversals.  These extend classes of total orientations introduced by Gioan \cite{gioan2007enumerating}, and Greene and Zaslavksy \cite{greene1983interpretation}, and we highlight a close connection with graphic and cographic Lawrence ideals.  We conclude with edge chromatic generalizations of the quantities presented, which allow for a new interpretation of the reliability polynomial for all probabilities, $p$ with $0 < p <1/2$.

\end{abstract}

\maketitle

\section{Introduction}
The Tutte polynomial is the most general bivariate polynomial which can be defined using deletion-contraction, and there are several known evaluations of the Tutte polynomial which count certain families of graph orientations.  In particular, Stanley \cite{stanley1973acyclic} showed that $T(2,0)$  counts the number of acyclic orientations of a graph  and Las Vergnas  \cite{las1980convexity} showed that $T(0,2)$ counts the number of strongly connected orientations of a graph.  These two facts are dual for planar graphs because acyclic orientations of a plane graph induce strongly connected orientations of its dual, while the Tutte polynomial of the dual graph is obtained by interchanging the variables.  Gessel and Sagan \cite{gessel1996tutte} investigated the Tutte polynomial using depth first search and applied their techniques to show the number of acyclic partial orientations of a graph is $2^gT(3,1/2)$.  We offer a short proof of this fact using deletion-contraction and demonstrate for the first time that the number of strongly connected partial orientations of a graph is $2^{|V|-1}T(1/2,3)$.

Gioan \cite{gioan2007enumerating} presented a unified framework for understanding the orientation based interpretations of the integer evaluations of $T(x,y)$ for $0 \leq x, y \leq 2$ using directed cut reversals, directed cycle reversals, and a convolution formula for the Tutte polynomial in terms of the cyclic flats of a graph.  In the process, Gioan introduced the notion of orientations with a quasi sink which give distinguished representatives for orientations modulo cut reversals and are counted by $T(1,2)$.  On the other hand, it was shown by Stanley \cite{stanley1980decompositions} that the indegree sequences of full orientations are counted by $T(2,1)$, and another interpretation of this quantity was given by Gioan as the number of full orientations modulo cycle reversals.  This quantity has yet another interpretation as the set of full orientations such that the cyclic part is minimal in the sense of Greene and Zaslavsky \cite[Corollary 8.2]{greene1983interpretation} or Bernardi \cite{bernardi2008tutte}, as these give distinguished representatives for the set of full orientations modulo cycle reversals.

Using a total order on the edges and a fixed reference orientation of our graph, we introduce cut minimal and cycle minimal partial orientations which extend Gioan's $q$-connected orientations and Greene and Zaslavsky's minimal orientations respectively.  Cut minimal partial orientations are more matroidal than Gioan's $q$-connected orientations, even in the case of full orientations, as they do not require the notion of a vertex.  The cut minimal partial orientations give unique representatives for the equivalence classes of partial orientations modulo cut reversals which we prove are enumerated by $2^{|V|-1}T(1,3)$.  Similarly, the cycle minimal partial orientations give distinguished representatives for the equivalence classes of partial orientations modulo cycle reversals which we prove are enumerated by $2^gT(3,1)$.

In \cite{backman2014riemann}, the author generalized Gioan's cycle reversal, cocycle reversal, and cycle-cocycle reversal systems to partial orientations by the addition of {\it edge pivots}.     It was demonstrated that two partial orientations have the same indegree sequence if and only if they are equivalent by cycle reversals and edge pivots.  We strengthen this result and introduce cycle-path minimal partial orientations which give distinguished representatives for the set of partial orientations with a fixed indegree sequence. 

Lawrence ideals are certain binomial ideals associated to lattices, and in \cite{drton2009lectures, mohammadi2013divisors, kateri2014family} the Lawrence ideals associated to the cut and cycle lattices were investigated in the context of combinatorial commutative algebra and algebraic statistics.  These ideals have distinguished Gr\"obner bases encoding cut reversals and cycle reversals, and we explain how these objects lend themselves to alternate proofs of existence and uniqueness for cut minimal and cycle minimal partial orientations.  

We conclude by describing {\it edge chromatic} generalizations of the objects counted in this paper where each oriented edge takes one of $k$ colors and each unoriented edge takes one of $l$ colors.  We give formulas for the number of such objects using the Tutte polynomial and we apply $(k,l)$-chromatic cut minimal partial orientations to give a new interpretation of the reliability polynomial for all probabilities between 0 and $1/2$.  

Our enumerative results follow from weighted deletion-contraction relations which seem to appear naturally in the context of the Potts model from physics, e.g. \cite{sokal2005multivariate}, and the $V$-polynomial from knot theory \cite{noble1999weighted}.  In \cite{sokal2005multivariate} we find that Sokal writes ``Let me conclude by observing that numerous specific evaluations of the Tutte
polynomial have been given combinatorial interpretations, as counting some set of
objects associated to the graph $G$ ...  It would be an interesting
project to seek to extend these counting problems to ``counting with weights''...''  Perhaps this paper is a step in the direction which Sokal envisioned.

\section{Acyclic and Strongly Connected Partial Orientations}\label{secacyclic}

All graphs considered in this paper will be undirected and connected although they may have multiple edges or loops.  We set $m = |E(G)|$ and $n=|V(G)|$.  The {\it genus} of a graph, $g$ is the quantity $m-n+1$.    A partial orientation of a graph is an orientation of a subset of the edges, and we say that the remaining edges are unoriented.  A partial orientation is {\it acyclic} if it contains no directed cycles, and {\it strongly connected} if it does not contain any directed cuts. A strongly connected full orientation of a graph is most commonly defined as containing a directed path between each ordered pair of vertices, but this is equivalent to the condition that the orientation does not contain any directed cuts.  An acyclic partial orientation of a loopless graph is one which can be extended to a full acyclic orientation, and a strongly connected partial orientation of a bridgeless graph is one which can be extended to a strongly connected full orientation \cite{farzad2006forced}.  Moreover, these objects are dual for plane graphs.

     \begin{figure}[h] \label{AS}
\centering
\includegraphics[height=5cm]{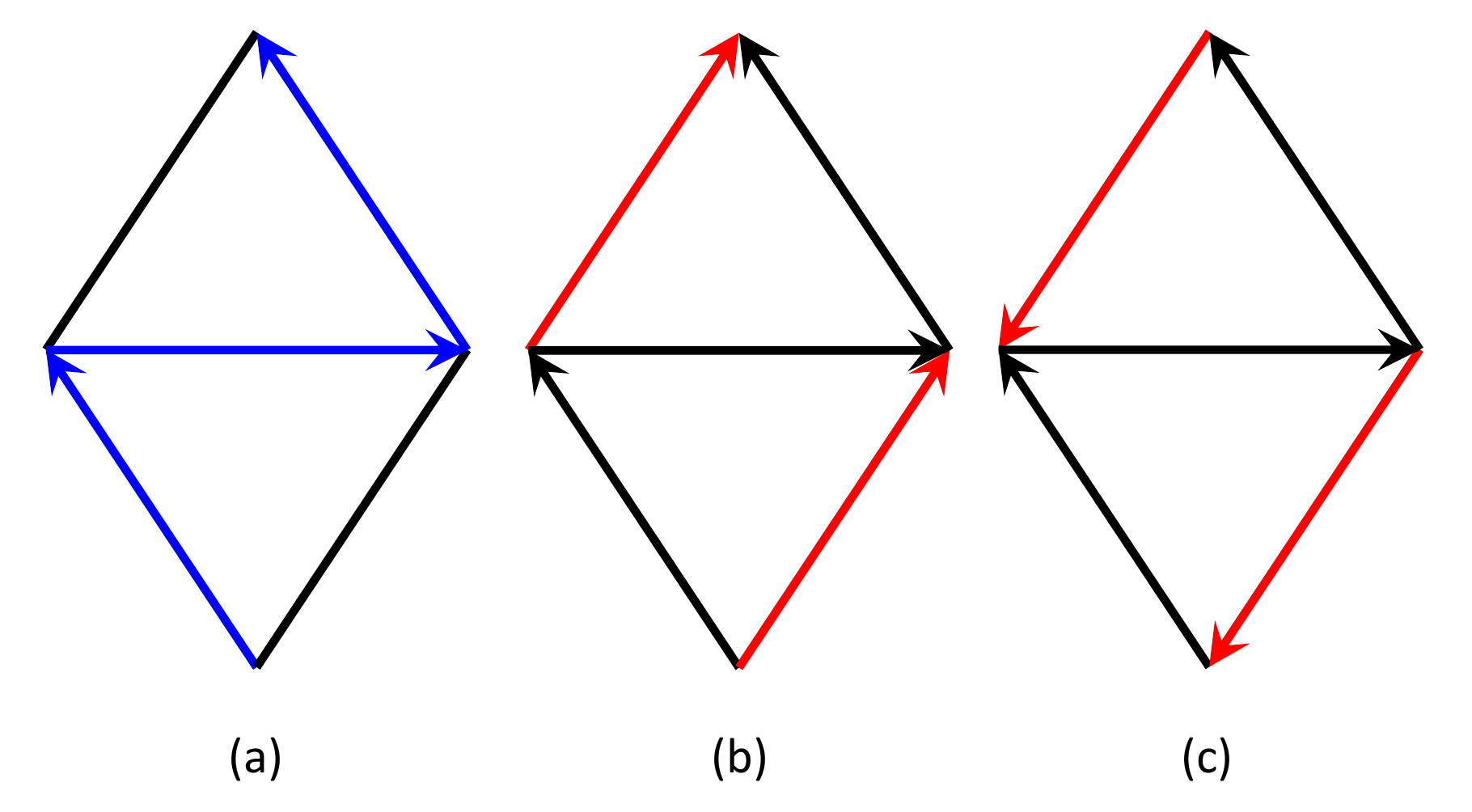}
\caption{ (a) A partial orientation which is both acyclic and strongly connected, (b) an acyclic full orientation obtained by orienting the remaining unoriented edges, and (c) a strongly connected full orientation obtained by orienting the remaining unoriented edges.}
\end{figure}

\begin{theorem}\label{acyclic}
The number of acyclic partial orientations of a graph is $2^gT(3,1/2)$.
\end{theorem}

\begin{proof}
Let $f(G)$ be the number of acyclic partial orientations of $G$.  Assuming that $e$ is not a loop, we claim that $f(G) = 2f(G \setminus e) + f(G/e).$  Given a partial orientation of $G\setminus e$ we can extend this to a partial orientation of $G$ in at least two ways:  we can leave $e$ unoriented or we can orient it in one of the two possible directions.  Suppose that both orientations of $e$ cause a cycle to appear.  This is true if and only if there is a directed path from $u$ to $v$ and a directed path from $v$ to $u$, contradicting the acyclicity of $\O$.  If there is no path from $u$ to $v$ or $v$ to $u$ then both orientations preserve acyclicity and this occurs precisely when the orientation obtained by contracting $e$ is acyclic, hence the relation above.  Now we need to check that the proposed function also satisfies this relationship:  $2^{g(G)}T_G(3,1/2) = 2 \cdot 2^{g(G \setminus e)}T_{G\setminus e}(3,1/2) + 2^{g(G/e)}T_{G / e}(3,1/2)$.  If $e$ is a bridge then $f(G) = 3f(G / e)$, and if $e$ is a loop, $f(G) = 2\cdot 1/2f(G/e)$.

\end{proof}

\begin{theorem}\label{strong}
The number of strongly connected partial orientations of a graph is $2^{n-1}T(1/2,3)$.  
\end{theorem}

\begin{proof}
Let $f(G)$ be the number of strongly connected partial orientations of $G$.  Assuming that $e$ is not a bridge, we claim that $f(G) = f(G \setminus e) + 2f(G/e).$  Given a partial orientation of $G\setminus e$ we can extend this to a strongly connected partial orientation of $G$ in at least two ways:  we can leave $e$ unoriented or we can orient it in at least one of the two directions.  Suppose that both orientations of e cause a directed cut to appear.  Let $(X,X^c)$ and $(Y,Y^c)$ be the two associated cuts.  It easy each to check that either $(X\cap Y, X^c \cup Y^c)$ or $(X\cap Y^c, X^c\cup Y)$ is already consistently oriented, a contradiction.  Both orientations of $e$ preserve strong connectedness if and only if the deletion of $e$ yields a strongly connected orientation, hence we obtain the desired relation.   Now we need to check that the proposed function satisfies the above recurrence:  $2^{|V(G)|-1}T_G(1/2,3) = 2^{|V(G\setminus e)|-1}T_{G \setminus e}(1/2,3)  + 2\cdot 2^{|V(G/e)|-1}T_{G / e}(1/2,3)$.
 Finally, if $G = e$ is a bridge  $f(G) = 2\cdot1/2f(G/e)$ and if $G = e$ is a loop, $f(G) = 3f(G \setminus e)$.
 \end{proof}
 
 We note that for plane graphs there is a duality relating the exponents appearing in front of the Tutte polynomial in the two previous theorems.  By Euler's formula, the genus $g$ of a plane graph is one less than the number of faces, which is one less than the number of vertices of the dual plane graph. 
 
In \cite{beck2011enumeration, beck2012weak} a different extension of Stanley's theorem is given.  Their setup is to start with a fixed acyclic partial orientation, also called an acyclic mixed graph, and count the number of ways this can be completed to a full acyclic orientation using an evaluation of a chromatic polynomial which they associate to the mixed graph.  Another pair of related papers is \cite{hopkins2011orientations, hopkins2012bigraphical}, where the authors use acyclic partial orientations for proving the $G$-Shi conjecture.  They prove that the regions of the bigraphical arrangement are in bijection with certain types of acyclic partial orientations which they call $A$-admissible.  They also prove that the number of regions of a generic bigraphical arrangement is $2^{n-1}T(3/2,1)$, a value which appears somewhat similar to the one obtained in Theorem \ref{acyclic}.

 \section{Cut Minimal, Cycle Minimal and Cycle-Path Minimal Partial Orientations}\label{secminimal}
 
In this section we are concerned with 3 different operations on partial graph orientations.  These operations are called {\it cut (cocycle) reversals}, {\it cycle reversals}, and {\it edge pivots}.  In a cut reversal, the edges in a directed cut are reversed.  Similarly, in a cycle reversal, the edges in a directed cycle are reversed.  Cut reversals and cycle reversals were introduced by Gioan \cite{gioan2008circuit} for full orientations, the former of which generalizes Mosesian's {\it pushing down} operation \cite{mosesian1972strongly} for acyclic orientations. Given a vertex $v$ incident to an unoriented edge $e$ and an edge $e'$ oriented towards $v$, an edge pivot is the operation of unorienting $e'$ and orienting $e$ towards $v$. The notion of an edge pivot was introduced by the author in \cite{backman2014riemann} for the study of partial orientations.  See Figure \ref{pivotalcycle-cocyclenoshadow} for an illustration of these operations.  Given a partial orientation $\O$, we will refer to its indegree sequence $D_{\O}$ which is a vector whose $i$th entry is the number of edges oriented towards the $i$th vertex.  We remark that for the study of divisors on graph we take the indegree minus one, but the -1 is not important for the considerations of this paper, hence we omit it.

In what follows, we work with a pair $(<,A)$ where $<$ is a total order on the edges and $A$ is a full orientation of the graph.  We should consider $A$ as a {\it reference orientation} of the graph to which we can compare other partial orientations. 
 
     \begin{figure}[h] 
\centering
\includegraphics[height=10cm]{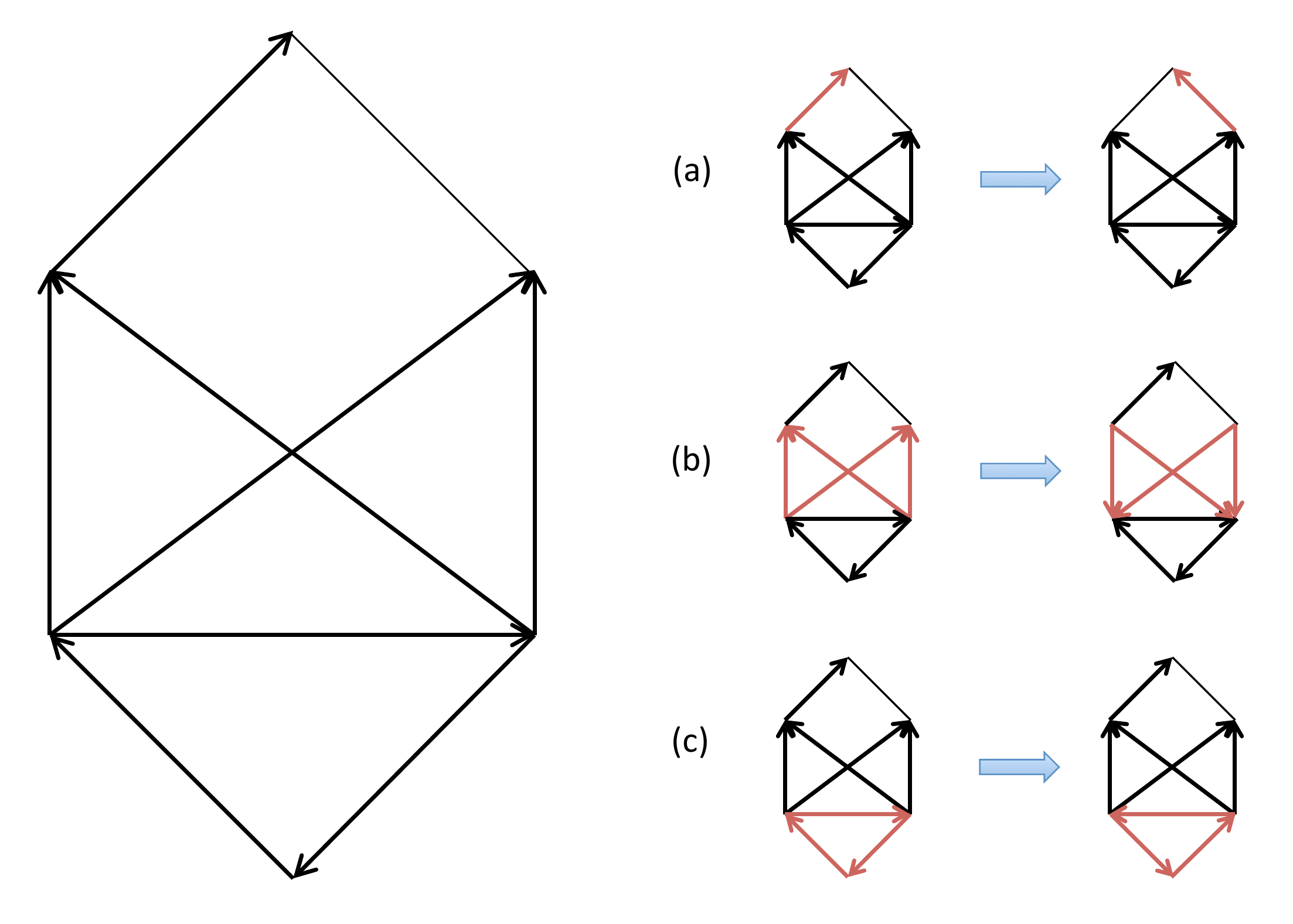}
\caption{ A partial orientation with (a) an edge pivot, (b) a cut (cocycle) reversal, and (c) a cycle reversal.}\label{pivotalcycle-cocyclenoshadow}
\end{figure}

 \begin{definition}
 A directed cut is minimal if its minimum edge is oriented in the same direction as in $A$, and nonminimal otherwise.  A partial orientation $\O$ of $G$ is {\it cut minimal} if every directed cut in $\O$ is minimal.
   \end{definition}
 
 We first explain how these objects generalize Gioan's orientations with a quasisink, known elsewhere as $q$-connected orientations or root connected orientations.  These are the full orientation such that there exists a vertex $q$ from which every other vertex is reachable by a directed path, and they generalize Mosesian's root connected acyclic orientation \cite{mosesian1972strongly} which are counted by $T(1,0)$, e.g. \cite{benson2010g}.  It is easy to see that this definition is equivalent to the statement that each directed cut is oriented away from $q$.  These orientations were rediscovered by An, Baker, Kuperberg and Shokrieh \cite{an2013canonical}, who observed that their associated divisors are the break divisors of Mikhalkin and Zharkov \cite{mikhalkin2006tropical} offset by a chip at $q$.
  To see how $q$-connected orientations can be realized as cut minimal orientation, take a spanning tree $T$ with all of the edges oriented away from $q$ and label them with with the integers from $1$ to $n-1$ so that each edge has a smaller label than it's descendants, e.g. by depth first search or breadth first search.  Extend this order and orientation of $T$ arbitrarily to a total order of the edges and full orientation of the graph $(<,A)$.  We claim that the $q$-connected orientations are precisely the cut minimal orientations with respect to $(<,A)$.  For any directed cut, the minimum edge in this cut is the smallest edge which also appears in $T$, therefore this cut is minimal if and only if it is oriented away from $q$.  We call any pair $(<,A)$ coming from such a construction, a {\it q-connected pair}.
 
     \begin{figure}[h] \label{pair}
\centering
\includegraphics[height=5cm]{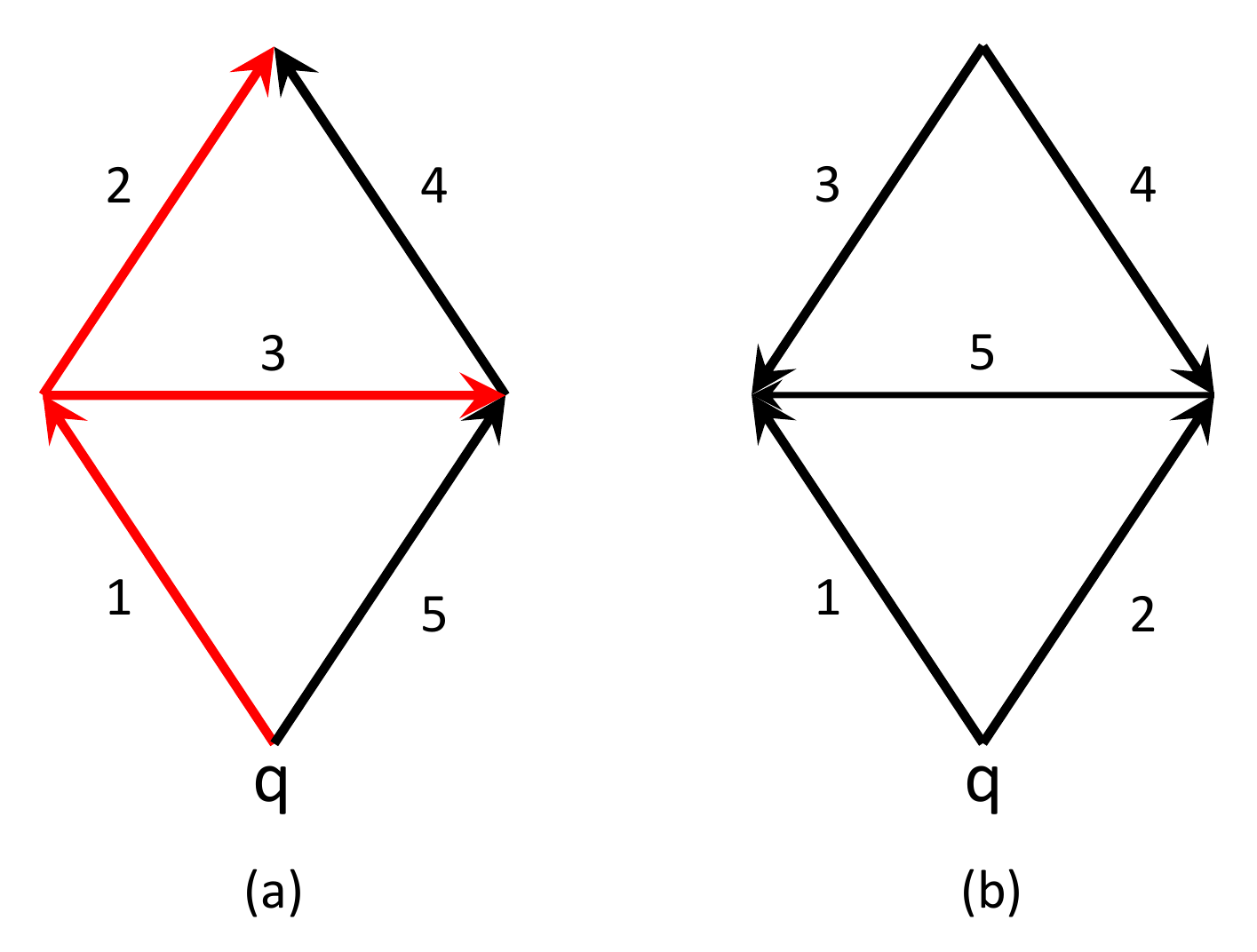}
\caption{ (a) A $q$-connected pair $(<_0,A_0)$ with a rooted spanning tree in red and (b) a pair $(<_1,A_1)$ which is not $q$-connected.}
\end{figure}

 \begin{theorem}\label{cut-min}
 Each partial orientation $\O$ is equivalent via cut reversals to a unique cut minimal orientation, which can be obtained greedily.
 \end{theorem}
 
 \begin{proof}
We proceed greedily by picking nonminimal directed cuts in $\O$ and reversing them.  We first claim that this process will terminate.  Supposing that this is not the case, we must without loss of generality return to $\O$.  Let $e$ be the minimum edge in $\O$ which was reversed in some cut $C$ before returning to $\O$.  This edge is minimal in $C$ as well as some other cut $C'$ which was reversed when $e$ was reoriented.  Therefore, either $C$ or $C'$ was minimal when it was reversed, a contradiction.  

To prove uniqueness, we should first prove that if two partial orientations are related by cut reversals then their symmetric difference decomposes as a disjoint collection of directed cuts from which the statement follows immediately.  If $\O$ and $\O'$ are related by cut reversals, then their corresponding indegree sequences $D_{\O}$ and $D_{\O'}$ are related by chip-firing, i.e. $D_{\O}-D_{\O'} = Qf$, where $Q$ is the Laplacian matrix and $f$ is some integer vector.  Because the kernel of the Laplacian is generated by the all one's vector, we may assume with out loss of generality that $f \geq0$ and $f(v)=0$ for some $v$.  Let $X$ be the support of $f$, then $(X,X^c)$ forms a directed cut in $\O$ which we can reverse and induct on the size of the support of $f$.

 \end{proof}
 
  \begin{figure}[h] \label{cut-minimal}
\centering
\includegraphics[height=5cm]{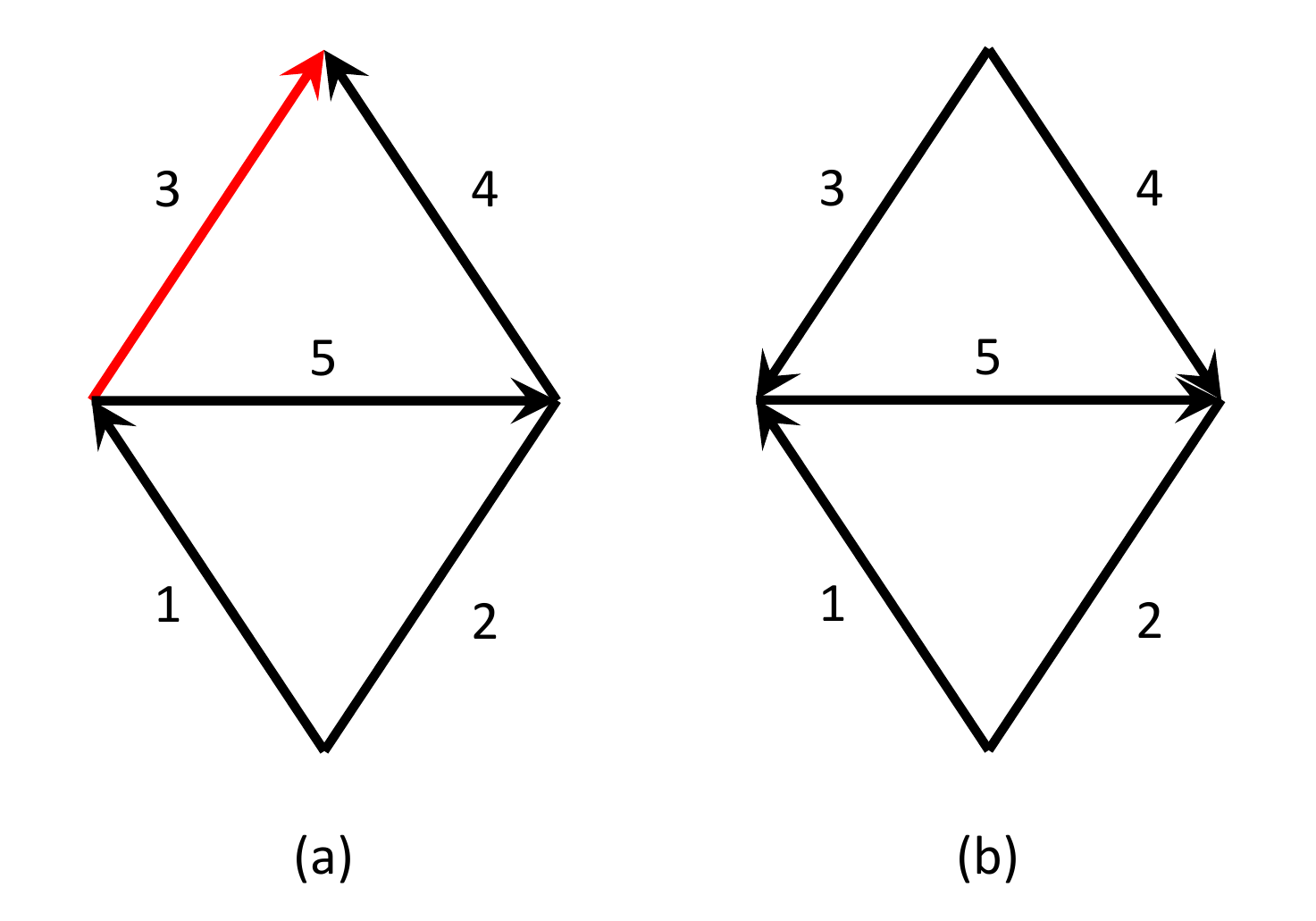}
\caption{ (a) A partial orientation which is not cut minimal with respect to $(<_1,A_1)$ (see figure \ref{pair}) because the minimum edge 3 in the directed cut $\{ 3,4\}$ is oriented oppositely from $A_1$ and (b) the cut minimal partial orientation obtained by reversing this directed cut.}
\end{figure}
 
 In the case of strongly connected full orientations, the following definition agrees with that of Greene and Zaslavsky \cite[Corollary 8.2]{greene1983interpretation}.  A different notion of a cycle-minimal full orientation, called an $\alpha$-minimal orientation appears in Bernardi \cite{bernardi2008tutte} which is constructed from a combinatorial map, i.e. a ribbon graph structure as opposed to a total order on the edges and a reference orientation. Unfortunately, providing a self contained definition of Bernardi's  $\alpha$-minimal orientations would take us beyond the scope of this paper.
 
 \begin{definition}
A directed cycle is minimal if its minimum edge is oriented in the same direction as in $A$, and nonminimal otherwise.  A partial orientation $\O$ of $G$ is {\it cycle minimal} if every directed cycle in $\O$ is minimal.
 \end{definition}
 
\begin{theorem}\label{cycle-min}
 Each partial orientation $\O$ is equivalent via cycle reversals to a unique cycle minimal partial orientation, which can be obtained greedily.
\end{theorem}

 \begin{proof}
 The proof is the same as (dual to) the proof of Theorem \ref{cut-min}.  We can greedily reverse nonminimal cycles until we eventually reach a cycle minimal partial orientation.  Uniqueness follows from the fact that if two partial orientations differ by cycle reversals, then their difference decomposes as an edge disjoint union of directed cycles.
 \end{proof}
 
 Greene and Zaslavsky \cite[Corollary 8.3]{greene1983interpretation} note that the number of strongly connected orientations of a plane graph $G$ such that all of the cycles are orientated counterclockwise is $T(0,1)$ since there exist pairs $(<,A)$ which makes this set the collection of cycle minimal strongly connected orientations.  It seems worth noting that the set of pairs $(<,A)$ which accomplish this task are in bijection with the set of $q$-connected pairs $(<,A)$ of the dual graph as $q$-connected acyclic orientations of a plane graph induce strongly connected orientations of the dual with no counterclockwise oriented cycles and vice versa.
  
 We can similarly define a cycle-cut minimal partial orientation to be one which is both cycle minimal and cut minimal.  Because the set of edges which belong to directed cycles and directed cuts is disjoint and preserved under cut and cycle reversals, we obtain the following corollary of Theorems \ref{cut-min} and Theorem \ref{cycle-min}.
 
 \begin{corollary}\label{cycle-cut minimal}
 Every partial orientation is equivalent by cycle reversals and cut reversals to a unique cycle-cut minimal partial orientation.
 \end{corollary}

We now employ cut minimal and cycle minimal partial orientations to obtain formulas for  the number of partial orientations modulo cut reversals and the number of partial orientations modulo cycle reversals.

\begin{theorem}\label{cut-min-Tutte}
The number of partial orientations of a graph $G$ modulo cut reversals is $2^{n-1}T(1,3)$.
\end{theorem}

\begin{proof}
By Theorem \ref{cut-min}, it is equivalent to prove that the number of cut minimal partial orientations, $f(G)$ is $2^{n-1}T(1,3)$.  Because this quantity is independent of $(<,A)$, we take this pair to be a $q$-connected pair.  Let $e = (q,v)$ be the minimum edge in $G$ with respect to $<$, and suppose that $e$ is not a loop or cut edge.  We claim that $f(G) = f(G \setminus e) + 2f(G/e)$.  First we note that there is a bijection between cut minimal orientations of $f(G)$ we with $e$ oriented in the opposite direction of $A$, i.e. towards $q$, and cut minimal orientations of $f(G \setminus e)$.  Also, there is a bijection between cut minimal orientations of $G$ with $e$ unoriented and cut minimal orientations of $f(G/e)$.  The same is true for cut minimal orientations of $f(G)$ with $e$ oriented in the same direction as in $A$, hence the given deletion-contraction relations holds.  For $e$ a cut edge, $f(G) = 2\cdot 1f(G/e)$ and if $e$ is a loop then $f(G) = 3f(G/e)$ from which the theorem follows.  
\end{proof}

\begin{theorem}\label{cycle-min-Tutte}
The number of partial orientations of a graph $G$ modulo cycle reversals is $2^gT(3,1)$.
\end{theorem}

\begin{proof}
By Theorem \ref{cycle-min}, it is equivalent to prove that the number of cycle minimal partial orientations, $f(G)$ is $2^gT(3,1)$.  We will show that $f(G) = 2f(G \setminus e) + f(G/e)$.  Let $e = (q,v)$ be neither a bridge nor a loop.  We first construct a pair $(<,A)$ by taking $e$ to be the minimum labeled edge and orient it towards $v$.  We then label the rest of the edges incident to $q$, $2$ through deg$(q)$ and orient them towards $q$.  We now extend this arbitrarily to a pair $(<,A)$.  We first observe that there is a bijection between cycle minimal orientations with $e$ oriented towards $q$ and cycle minimal orientations of $G / e$.  There is a bijection between cycle minimal orientations of $G$ with $e$ oriented towards $v$ and cycle minimal orientations of $G \setminus (q,v)$.  The same is true for $(q,v)$ unoriented, hence the deletion-contraction holds.  Finally, if $e$ is a cut edge, then $f(G) = 3f(G \setminus e)$ and if $e$ is a loop, $f(G) = 2 
\cdot 1 f(G/e)$.
\end{proof}

   We now show how an arbitrary pair $(<,A)$ naturally picks out a distinguished partial orientation representing a given indegree sequence.  We call a directed path with the addition of an unoriented edge incident to the terminal vertex a {\it half open path}.  For discussing the orientation of the edges in a half open path, we imagine that the unoriented edge has the orientation which would allow us to extend the given path.  Note that an unoriented edge may have different orientations depending on which half open path we are considering.  We say that a {\it half open path} is minimal if its minimum edge is oriented in the same direction as in $A$.  If a half open path is nonminimal, we can perform a sequence of edge pivots to reverse the orientation of its edges.  We call such a sequence of edge pivots, a {\it Jacob's ladder cascade} \cite[Definition 3.2]{backman2014riemann}, but for the sake of brevity, we will usually refer to this operation as a {\it cascade}.   See Figure \ref{jacob'sladder6}.
   
 \begin{figure}[h]

\centering
\includegraphics[height=7.5cm]{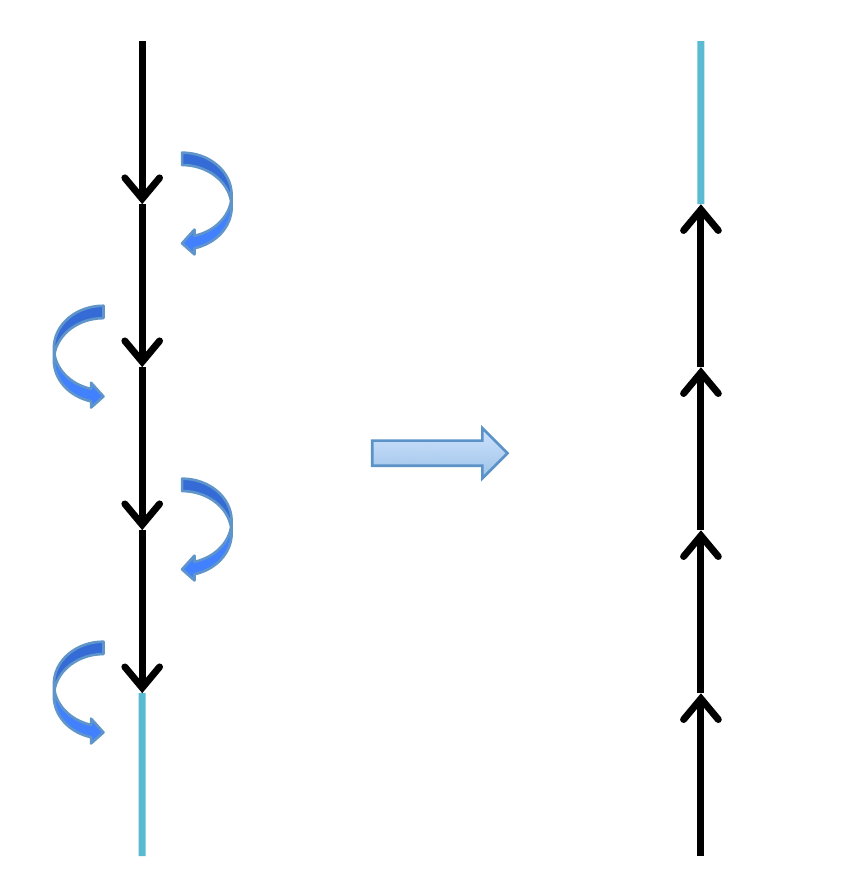}
\caption{Two half-open paths related by a Jacob's ladder cascade.}\label{jacob'sladder6}

\end{figure}  
   
   \begin{definition}
   A partial orientation is cycle-path minimal if every directed cycle is minimal and every half-open path is minimal.
   \end{definition}
   
      As shown in the author's previous paper \cite[Lemma 3.1]{backman2014riemann}, two partial orientations have the same indegree sequence if and only if one can be obtained from the other a collection of edge pivots and cycle reversals.  Before demonstrating the existence and uniqueness of cycle-path minimal partial orientations, we will need the following strengthening of this result.
      
     \begin{lemma}\label{disjoint}
     Any two partial orientations which have the same indegree sequence are related by an edge disjoint union of cycle reversals and Jacob's ladder cascades. 
     \end{lemma}

\begin{proof}

Let $\O$ and $\O'$ be two partial orientations which are related by edge pivots and cycle reversals.  Let $S$ be the set of edges in $\O$ with different orientations in $\O'$, and $e$ be an unoriented edge belonging to $S$ which is oriented toward $v$ in $\O'$.  Begin a directed walk backwards from $v$ in $S$ along edges which are oriented oppositely in $\O'$.  Eventually this path will either return to some vertex already visited, giving a directed cycle $C$, or it will terminate at some vertex incident to an edge which would extend the path, but is unoriented in $\O'$.  In the former case we can reverse this directed cycle and induct on $|S|$.  In the latter case we have a half open path such that we can perform the corresponding cascade and again induction on $|S|$.

\end{proof}

   \begin{theorem}\label{cycle-path-min}
 Each partial orientation $\O$ is equivalent via cycle reversals and edge pivots to a unique cycle-path minimal partial orientation.
    \end{theorem}

   \begin{proof}
We first prove existence.  Given a partial orientation $\O$, we may perform cascades and cycle reversals greedily to eliminate all nonminimal half-open paths and cycles.  Supposing that this process does not terminate, we must, without loss of generality, return to $\O$.  As in previous arguments we restrict attention to the minimum labeled edge  in $\O$ (possibly unoriented) which was part of a cascade or a cycle reversal before returning to $\O$ and we find that it could not have changed orientation twice, a contradiction.

Suppose that there exists two cycle-path minimal partial orientations, which are related by edge pivots and cycle reversals.  By Lemma \ref{disjoint} they are related by an edge disjoint union of cascades and cycle reversals.  Any half open path appearing in this set is minimal, contradiction the fact that the half open path obtained by a cascade in $\O'$ is also minimal.  Therefore the set of cascades is empty and we reduce to Theorem \ref{cycle-min}.

     \end{proof}
    
  As mentioned above, any two partial orientations with the same indegree sequence are equivalent by cycle reversals and edge pivots, hence cycle-path minimal orientations provide distinguished partial orientations with a given indegree sequence.  The number of indegree sequences of full orientations is counted by $T(2,1)$, but for partial orientations there does not seem to be a simple formula in terms of the Tutte polynomial.  One basic reason is that there exist different trees on the same number of vertices with different numbers of indegree sequences coming from partial orientations.  For example, the path on three edges has 21 indegree sequences while the star on three edges has 20. 
  
  \section{Lawrence Ideals}\label{secbinomial}
  
Mohammadi and Shokrieh \cite{mohammadi2013divisors} and Kateri, Mohammadi,  and Sturmfels \cite{kateri2014family} investigated the Lawrence ideals associated to the cut and cycle lattices respectively and proved that these ideals have universal Gr\"obner bases which encode directed cut reversals and cycle reversals.  The latter set of authors were motivated by a connection with algebraic statistics described in \cite{drton2009lectures}.   We briefly recall the construction of these ideals, their distinguished Gr\"obner bases, and explain how they lend themselves to alternate proofs of Theorem \ref{cut-min} and Theorem \ref{cycle-min}.

Given a lattice $\Lambda \subset \mathbb{Z}^m$, we can associate a {\it Lawrence ideal}, $I_{\Lambda} = < x^{u^+}y^{u^-}-x^{u^-}y^{u^+}: u \in \Lambda>$, where $u^+$ and $u^-$ are the positive and negative parts of the vector $u$ respectively.  Fix a pair $(<,A)$.  Let $e_i$ be the $i$-th edge with respect to $<$ which is oriented as in $A$ and $\bar e_i$ to be the $i$-th edge with the opposite orientation.  Let $R= K[x_1, \dots, x_m,y_1,\dots, y_m]$ be the polynomial ring in $2m$ variables.   We define a map $\phi$ from oriented edges to variables: $\phi(e_i) = x_i$, $\phi({\bar e}_i) = y_i$, and we extend this to partial orientations by setting $\phi(\O_1 \sqcup \O_2)= \phi(O_1)\phi(O_2)$.  Given a directed cut $B$, we define $\bar B$ to be the reverse directed cut, and given a directed cycle $C$, we define $\bar C$ to be the reverse directed cycle.  We let

$$\mathcal{B} = \{ \phi(B)-\phi({\bar B}) : B {\rm \,\,is \,\,a \,\,directed \,\,cut} \}$$

$$\mathcal{C} = \{ \phi(C)-\phi({\bar C}) : C {\rm \,\, is\,\, a\,\, directed\,\, cycle} \}$$

\

  It was shown in \cite[Proposition 7.8]{mohammadi2013divisors} and \cite[Lemma 3.1]{kateri2014family} that the set of binomials in $B$ and $C$ associated to minimal cuts and simple cycles form universal Gr\"obner bases for the Lawrence ideal associated to the cut and cycle lattices respectively.  In particular this shows that we can change our input data $(<,A)$ by permuting the first $m$ and second $m$ variables similarly, or interchanging $x_i$'s for $y_i$'s, and our cuts and cycles remain a reverse lexicographic Gr\"obner basis.  Encoding a partial orientation as a squarefree monomial and interpreting division by elements of $B$ and $C$ as cut reversals and cycle reversals, we obtain Theorem \ref{cut-min} and Theorem \ref{cycle-min} because division by a Gr\"obner basis always yields a unique remainder.

Similar to the way we are able to encode directed cut reversals and cycle reversal using binomials, we can also describe Jacob's ladder cascades.  Given a half open path $P$, we define $\bar P$ to be the half open path obtained by the corresponding cascade, and take

$$\mathcal{P} = \{ \phi(P)-\phi({\bar P}) : P {\rm \,\,is\,\, a\,\, half\,\, open\,\, path} \}.$$ 

\

By Theorem \ref{cycle-path-min}, we find that the set $C \cup P$ is reverse lexicographic Gr\"obner basis for the ideal which it generates.  We conjecture that this set is also a universal Gr\"obner basis, although they are not obviously Lawrence ideals so {\it a priori} the seemingly relevant \cite[Theorem 7.1]{sturmfels1996grobner} of Sturmfels does not apply.

\section{Edge Chromatic Extensions and the Reliability Polynomial}\label{secchromatic}
 
In this section we describe ``edge chromatic" generalizations of the previous enumerations and illustrate a relationship with the reliability polynomial.  We define a {\it (k,l)-chromatic partial orientation} to be a partial orientation obtained by assigning each of the oriented edges one of $k$ colors and each of the unoriented edges one of $l$ colors.  We leave it to the reader to verify the following generalizations of our previous results.
 
 \begin{itemize}
 
 \item The number of $(k,l)$-chromatic acyclic partial orientations of a graph is 
 $$ k^{n-1}(k+l)^gT({2k+l \over k},{l \over k+l}).$$
 
 \item The number of $(k,l)$-chromatic strongly connected partial orientations of a graph is 
 $$(k+l)^{n-1}k^gT({l \over k+l}, {2k+l \over k}).$$ 
 
  \item  The number of $(k,l)$-chromatic cycle minimal partial orientations of a graph $G$ is 
  $$k^{n-1}(k+l)^gT({2k+l \over k},1).$$
  
   \item The number of $(k,l)$-chromatic cut minimal partial orientations of a graph is 
 $$(k+l)^{n-1}k^gT(1,{2k+l \over k}).$$
  
 \end{itemize}
 
 We note that if we plug in $(k,l) = (1,0)$, we recover previously known results for full orientations.

 The reliability polynomial of an undirected graph is $R(p)= (1-p)^{n-1}p^{g}T(1, {1\over p})$.  If we remove each edge in the graph with probability $p$, then $R(p)$ is the probability that our remaining subgraph is connected and spans the vertices.  For $p$ a rational number, this is easy to prove using weighted deletion-contraction techniques, e.g. as illustrated in earlier proofs, and the general statement follows by continuity.  We now show that $(k,l)$-chromatic cut minimal partial orientations of a graph allow for a different interpretation of $R(p)$ for all probabilities with $0 < p < {1 \over 2}$.  
 
 \begin{theorem}\label{reliable}
 Let $p$ be a rational probability between 0 and 1/2, and write $p= {k \over 2k+l}$ with $k$ and $l$ positive integers.   The evaluation of the reliability polynomial, $R(p)$ is the probability that a randomly chosen $(k,l)$-chromatic partial orientation is cut-minimal.
\end{theorem} 

\begin{proof}
The probability that a randomly chosen $(k,l)$-chromatic partial orientation is cut-minimal is  $${ (k+l)^{n-1}k^gT(1,{{2k+l \over k}})\over (2k+l)^m} =  $$
$$({k+l\over 2k+l})^{n-1} ({k\over 2k+l})^g T(1,{2k+l \over k} )= $$
$$(1-p)^{n-1}p^{g}T(1, {1\over p}). $$
\end{proof}

 \begin{corollary}
If we orient each edge in a graph $G$ with probability $p$ in either direction and leave it unoriented with probability $1- 2p$, then $R(p)$ is the probability that the randomly chosen partial orientation is cut minimal.
 \end{corollary}
 
 \begin{proof}
 This follows from Theorem \ref{reliable} by continuity.
 \end{proof}

 In \cite{mohammadi2014divisors} Mohammadi investigates system reliability in the context of directed cuts and partial orientations via combinatorial commutative algebra.  It would be interesting if one could apply the setup from her paper to give an algebraic reinterpretation of the previous result.

 \section*{Acknowledgments}
 Thanks to Yan X Zhang for helpful discussions about strongly connected partial orientations and Bernardi's $\alpha$-minimal orientations, and to Sam Hopkins and Dave Perkinson for explaining their work to me.  Additional thanks to Sam for suggesting that I investigate the relationship between my work and the reliability polynomial.  Thanks to Farbod Shokrieh for pointing out that the ideals described in section \ref{secbinomial} are Lawrence ideals, and to Fatemeh Mohammadi for providing further references.  I would like to express my gratitude to the Center for Applications of Mathematical Principles at the National Institute for Mathematical Sciences in Daejeon, South Korea where I was supported while this work was completed during the Summer 2014 Program on {\it Applied Algebraic Geometry}.

\bibliography{partialorientations}

\begin{thebibliography}{10}

\bibitem{an2013canonical}
Yang An, Matthew Baker, Greg Kuperberg, and Farbod Shokrieh.
\newblock Canonical representatives for divisor classes on tropical curves and
  the matrix-tree theorem.
\newblock {\em arXiv preprint arXiv:1304.4259}, 2013.

\bibitem{backman2014riemann}
Spencer Backman.
\newblock Riemann-roch theory for graph orientations.
\newblock {\em arXiv preprint arXiv:1401.3309}, 2014.

\bibitem{beck2012weak}
Matthias Beck, Daniel Blado, Joseph Crawford, Taina Jean-Louis, and Michael
  Young.
\newblock On weak chromatic polynomials of mixed graphs.
\newblock {\em arXiv preprint arXiv:1210.4634}, 2012.

\bibitem{beck2011enumeration}
Matthias Beck, Tristram Bogart, and Tu~Pham.
\newblock Enumeration of {G}olomb rulers and acyclic orientations of mixed
  graphs.
\newblock {\em arXiv preprint arXiv:1110.6154}, 2011.

\bibitem{benson2010g}
Brian Benson, Deeparnab Chakrabarty, and Prasad Tetali.
\newblock G-parking functions, acyclic orientations and spanning trees.
\newblock {\em Discrete Mathematics}, 310(8):1340--1353, 2010.

\bibitem{bernardi2008tutte}
Olivier Bernardi.
\newblock Tutte polynomial, subgraphs, orientations and sandpile model: new
  connections via embeddings.
\newblock {\em The Electronic Journal of Combinatorics}, 15(1):R109, 2008.

\bibitem{drton2009lectures}
Mathias Drton, Bernd Sturmfels, and Seth Sullivant.
\newblock {\em Lectures on Algebraic Statistics}.
\newblock Springer, 2009.

\bibitem{farzad2006forced}
Babak Farzad, Mohammad Mahdian, Ebadollah~S Mahmoodian, Amin Saberi, and Bardia
  Sadri.
\newblock Forced orientation of graphs.
\newblock {\em Bulletin of Iranian Mathematical Society}, 32(1):79--89, 2006.

\bibitem{gessel1996tutte}
Ira~M Gessel and Bruce~E Sagan.
\newblock The tutte polynomial of a graph, depth-first search, and simplicial
  complex partitions.
\newblock {\em Electron. J. Combin}, 3(2):R9, 1996.

\bibitem{gioan2007enumerating}
Emeric Gioan.
\newblock Enumerating degree sequences in digraphs and a cycle--cocycle
  reversing system.
\newblock {\em European Journal of Combinatorics}, 28(4):1351--1366, 2007.

\bibitem{gioan2008circuit}
Emeric Gioan.
\newblock Circuit-cocircuit reversing systems in regular matroids.
\newblock {\em Annals of Combinatorics}, 12(2):171--182, 2008.

\bibitem{greene1983interpretation}
Curtis Greene and Thomas Zaslavsky.
\newblock On the interpretation of {W}hitney numbers through arrangements of
  hyperplanes, zonotopes, non-radon partitions, and orientations of graphs.
\newblock {\em Transactions of the American Mathematical Society},
  280(1):97--126, 1983.

\bibitem{hopkins2012bigraphical}
Sam Hopkins and David Perkinson.
\newblock Bigraphical arrangements.
\newblock arXiv preprint arXiv:1212.4398 Forthcoming \emph{Trans. Amer. Math.
  Soc.}, December 2012.

\bibitem{hopkins2011orientations}
Sam Hopkins and David Perkinson.
\newblock Orientations, semiorders, arrangements, and parking functions.
\newblock {\em Electronic Journal of Combinatorics, 19(4)}, 2012.

\bibitem{kateri2014family}
Maria Kateri, Fatemeh Mohammadi, and Bernd Sturmfels.
\newblock A family of quasisymmetry models.
\newblock {\em arXiv preprint arXiv:1403.0547 Forthcoming \emph{Journal of
  Algebraic Statistics}}, 2014.

\bibitem{las1980convexity}
Michel Las~Vergnas.
\newblock Convexity in oriented matroids.
\newblock {\em Journal of Combinatorial Theory, Series B}, 29(2):231--243,
  1980.

\bibitem{mikhalkin2006tropical}
Grigory Mikhalkin and Ilia Zharkov.
\newblock Tropical curves, their {J}acobians and theta functions.
\newblock In {\em Curves and abelian varieties}, volume 465 of {\em Contemp.
  Math.}, pages 203--230. Amer. Math. Soc., Providence, RI, 2008.

\bibitem{mohammadi2014divisors}
Fatemeh Mohammadi.
\newblock Divisors on graphs, orientations, syzygies, and system reliability.
\newblock {\em arXiv preprint arXiv:1405.7972}, 2014.

\bibitem{mohammadi2013divisors}
Fatemeh Mohammadi and Farbod Shokrieh.
\newblock Divisors on graphs, binomial and monomial ideals, and cellular
  resolutions.
\newblock {\em arXiv preprint arXiv:1306.5351}, 2013.

\bibitem{mosesian1972strongly}
KM~Mosesian.
\newblock Strongly basable graphs.
\newblock In {\em Akad. Nauk. Armian. SSR Dokl}, volume~54, pages 134--138,
  1972.

\bibitem{noble1999weighted}
Steven~D Noble and Dominic~JA Welsh.
\newblock A weighted graph polynomial from chromatic invariants of knots.
\newblock In {\em Annales de l'institut Fourier}, volume~49, pages 1057--1087.
  Institut Fourier, 1999.

\bibitem{sokal2005multivariate}
Alan~D Sokal.
\newblock The multivariate {T}utte polynomial (alias {P}otts model) for graphs
  and matroids.
\newblock {\em Surveys in Combinatorics}, 327:173--226, 2005.

\bibitem{stanley1973acyclic}
Richard~P Stanley.
\newblock Acyclic orientations of graphs.
\newblock {\em Discrete Mathematics}, 5(2):171--178, 1973.

\bibitem{stanley1980decompositions}
Richard~P Stanley.
\newblock Decompositions of rational convex polytopes.
\newblock {\em Ann. Discrete Math. v6}, pages 333--342, 1980.

\bibitem{sturmfels1996grobner}
Bernd Sturmfels.
\newblock {\em Gr\"obner Bases and Convex Polytopes}.
\newblock American Mathematical Society Providence, 1996.

\end{thebibliography}

\

School of Mathematics, Georgia Institute of Technology

 Atlanta, Georgia 30332-0160, USA

 email address: {\it spencerbackman@gmail.com}

\end{document}